\documentclass[12pt,a4paper,reqno]{amsart}
\usepackage[left=3truecm,right=2.5truecm,top=2.5truecm,bottom=2.5truecm]{geometry}
\usepackage{amssymb, amsmath, tikz}
\usepackage{mathrsfs}
\usepackage{bbm}
\usepackage{cite}
\usepackage{graphicx}
\def\bc{\begin{center}}
\def\ec{\end{center}}
\def\be{\begin{equation}}
\def\ee{\end{equation}}

\def\N{\mathbb N}

\def\R{\mathbb R}

\newtheorem{lem}{Lemma}[section]

\newtheorem{pro}[lem]{Proposition}
\newtheorem{thm}[lem]{Theorem}

\newtheorem{cor}[lem]{Corollary}
\theoremstyle{remark}

\numberwithin{equation}{section}

\allowdisplaybreaks
\begin{document}

\title[Metrical properties for continued fractions of formal Laurent series]
{Metrical properties for continued fractions of formal Laurent series}

\author{Hui Hu}
\address{School of Mathematics and Information Science, Nanchang Hangkong University, Nanchang, Jiangxi 330063, PR China\\
and
Department of Mathematics and Statistics,
La Trobe University,
Bendigo 3552,
Australia}
\email{hh5503@126.com}

\author{Mumtaz Hussain}
\address{Department of Mathematics and Statistics,
La Trobe University,
Bendigo 3552,
Australia}
\email{m.hussain@latrobe.edu.au}

\author{Yueli Yu }
\address{School of Mathematics and Statistics, Wuhan University, Wuhan, Hubei 430072, PR China}
\email{yuyueli@whu.edu.cn}

\begin{abstract}
Motivated by recent developments in the metrical theory of continued fractions for real numbers concerning the growth of consecutive partial quotients, we consider its analogue over the field of formal Laurent series.  Let $A_n(x)$ be the $n$th partial quotient of the continued fraction expansion of $x$ in the field of formal Laurent series. We consider  the sets of  $x$ such that
$\deg A_{n+1}(x)+\cdots+\deg A_{n+k}(x)~\ge~\Phi(n)$ holds for infinitely many $n$ and for all $n$ respectively, where $k\ge1$ is an integer and $\Phi(n)$ is a positive function defined on $\mathbb{N}$.  We determine the size of these sets in terms of Haar measure and Hausdorff dimension.
\end{abstract}


\maketitle

\addtocounter{section}{0}

\section{Introduction and statement of results}
The continued fraction expansion of a real number is an alternative (and efficient) way to the decimal representation of a real number.  Every irrational  $x\in  (0,1)$ can be uniquely expressed as a simple continued
fraction expansion as follows
\begin{equation*}
x=\frac{1}{a_{1}(x)+\displaystyle{\frac{1}{a_{2}(x)+\displaystyle{\frac{1}{
a_{3}(x)+_{\ddots }}}}}}:=[a_{1}(x),a_{2}(x),a_{3}(x),\ldots ]
\end{equation*}
where $a_{n}(x)$ are positive integers and are called the partial quotients of $x$. The metrical theory of continued fractions concerns the quantity study of growth rate of partial quotients.   This expansion can be induced by the Gauss map $T_G:[0,1)\rightarrow [
0,1)$ defined as
\begin{equation*}
T_G(0):=0,\quad T_G(x):=\frac{1}{x}\  \mathrm{ (mod}\ 1),\quad \mathrm{for}\ x\in
(0,1), \label{GaussMap}
\end{equation*}
with $a_{1}(x)=\lfloor \frac{1}{x}\rfloor $, where $\lfloor .\rfloor $
represents the floor function and $a_{n}(x)~=~a_{1}(T_G^{n-1}(x))$ for $n\geq 2$.

Let $m\geq 1$ and $\Phi :\mathbb{N}\rightarrow (1,\infty)$ be a positive function. Define the set
 \begin{align*}\label{KKWW}
\mathcal{D}_{m}(\Phi)=\left\{x\in[0, 1): \prod_{i=1}^ma_{n+i}(x)\geq \Phi(n) \ \
\mathrm{for \ infinitely \ many \ }n\in\mathbb{N}\right\}.
\end{align*}
The well-known  Borel-Bernstein theorem \cite{Be_12, Bo_12} states that the Lebesgue measure of $\mathcal{D}_{1}(\Phi )$ is  either zero or full according to the convergence or divergence of the series ${\sum_{n=1}^{\infty }}1/\Phi (n)$ respectively.
For rapidly growing function $\Phi$, the Borel-Bernstein theorem does not give any conclusive information other than Lebesgue measure zero. To distinguish between such sets Hausdorff dimension is an appropriate tool. The Hausdorff dimension of the set $\mathcal{D}_{1}(\Phi )$
has been comprehensively determined by Wang-Wu \cite{WaWu08}.

Motivation for considering the growth of product of consecutive partial quotients arose from the works of Kleinbock-Wadleigh \cite{KlWad16}  where they considered improvements to Dirichlet's theorem.They defined the set of $\phi$-Dirichlet improvable numbers for some function $\phi$ (see Section \ref{Remark Dirichlet}) and proved that the Lebesgue measure of the set of $\phi$-Dirichlet non-improvable numbers is equivalent to the Lebesgue measure of the set $\mathcal{D}_{2}(\Phi)$ with $\Phi(n)=\frac{1}{1-b^n\phi(b^n)}-1$ for some $b>1$.
In particular, it was proved that the Lebesgue measure of the set of $\phi$-Dirichlet non-improvable numbers is zero or full if the series $\sum_{n}\frac{\log{\Phi}(n)}{{\Phi}(n)}$ converges or diverges respectively, see \cite[Theorem 3.6]{KlWad16} and \cite[Corollary 3.7]{KlWad16}.
The Hausdorff measure of this set was later established in \cite{HKWW, BHS}.   Very recently the Lebesgue measure and the Hausdorff dimension of  $\mathcal{D}_{m}(\Phi)$, for any $m\geq 2$ has been determined by Huang-Wu-Xu \cite{HuWuXu}. We refer the reader to \cite{BBH1, BBH2} for a comparison between the sizes of the classical well-approximable set with the set of Dirichlet non-improvable numbers.

These recent developments on the metrical theory of continued fractions for real numbers motivated the study of the analogous theory  for the continued fractions over the field of formal Laurent series. Let $\mathbb{F}_q$ be a finite field with $q$ elements and  $\mathbb{F}_q((z^{-1}))$ denotes the field of all formal Laurent series
$x=\sum_{n=v}^{\infty} c_nz^{-n}$ with coefficients $c_n\in \mathbb{F}_q$. If  $x=\sum_{n=v}^{\infty} c_nz^{-n}$ with $c_v\ne 0$ to be the first non-zero coefficient in the expansion of $x$ then the valuation (or norm) of $x$ is define by
\[|0|_{\infty}:=0,\quad  |x|_{\infty}:=q^{-v}.\]
It is well known that this valuation is  non-Archimedean. The topology induced by this norm make $\mathbb{F}_q((z^{-1}))$ locally compact and the ring of polynomials $\mathbb{F}_q[z]$ discrete in the field. Thus, we can think of $\mathbb{F}_q((z^{-1}))$ analogous to the set of real numbers $\R$ and $\mathbb{F}_q[z]$ akin to the set of integers $\mathbb Z$. Note that $\mathbb{F}_q((z^{-1}))$ is a complete metric space under the metric $\rho$ defined by $\rho(x,y)=|x-y|_{\infty}$.   Let $I$ be the valuation ideal of  $\mathbb{F}_q((z^{-1}))$, that is,
\[I=\left\{x\in \mathbb{F}_q((z^{-1})):|x|_{\infty}< 1 \right\}=\left\{ \sum_{n=1}^{\infty} c_nz^{-n}: c_n\in \mathbb{F}_q \right\}.\]
Let $\nu$ be the normalised (to 1) Haar measure on $I$. For $x=\sum_{n=v}^{\infty} c_nz^{-n}\in \mathbb{F}_q((z^{-1}))$, we call
$[x]=\sum_{n=v}^{0}c_nz^{-n}$ the integer part of $x$ and $\{x\}=\sum_{n=1}^{\infty}c_nz^{-n}$  to be the fractional part of $x$.

As in the real case, consider the Gauss transformation  $T: I\to I$ defined by
\begin{equation}\label{Gauss transformation in formal series}
 T(x):=\frac{1}{x}-\left[\frac{1}{x}\right],~\quad T(0):=0.
\end{equation}

Then each $x\in I$  has a finite or infinite continued fraction expansion induced by $T$,

\begin{equation*}
x=\frac{1}{A_{1}(x)+\displaystyle{\frac{1}{A_{2}(x)+\displaystyle{\frac{1}{
A_{3}(x)+_{\ddots }}}}}}:=[A_1(x), A_2(x), \ldots],
\end{equation*}
where the partial quotients $A_i(x)$ are polynomials of strictly positive degree  defined by
\[A_i(x)=\left[\frac{1}{T^{i-1}(x)}\right],\quad i\ge1.\]
This form of continued fraction induced from the Gauss map was first introduced by Artin \cite{Artin}, see also Berth\'{e} and Nakada \cite{BertheNakada}.

As in the case of real numbers, the metrical theory of continued fractions of formal Laurent series can be used to prove many Diophantine approximation results such as the analogues of the Borel-Bernstein theorem, Khintchine theorem, Jarn\'ik theorem and so on. 
The focus, at a fundamental level,  has been on the set
\[\mathcal F_1(\Phi):=\left\{x\in I: \deg A_n(x)\ge \Phi(n) \ \mathrm{for \ infinitely \ many \ }n\in\mathbb{N}\right\}.\]
The Haar measure, denoted as $\nu$,  of $\mathcal F_1(\Phi)$ was obtained by Niederreiter in \cite{Nied88}, see also\cite[Theorem 2.4]{Fuchs2002}, proving that the Haar measure of the set $\mathcal F_1(\Phi)$ is zero (respectively full) if the series $\sum_{n\ge1}q^{-\Phi(n)}$ converges (respectively diverges).
The Hausdorff dimension of this set was completely determined for any function $\Phi$ in \cite{Hu-Wang-Wu-Yu}. We refer the reader to \cite{BugeaudZhang, Kristensen2003, GangulyGhosh, SXJ, ZhangExact} for more related metrical (or distribution of digits) results over formal Laurent series.

Taking inspirations from the study of the growth of the product of consecutive partial quotients for the real numbers we initiate studying the growth of consecutive partial quotients over the field of formal Laurent series. Let $k\geq 1$.  Define the set
\begin{equation*}  \label{def-of -Ekvarphi}
\mathcal F_k(\Phi):=\left\{x\in I: \sum_{i=1}^k\deg A_{n+i}(x)\ge \Phi(n) \ \mathrm{for \ infinitely \ many \ }n\in\mathbb{N}\right\}.
\end{equation*}

There are several natural justifications for the consideration of this set. For any $n\geq 1$, consider the $n$th convergents of $x$
\[\frac{P_n(x)}{Q_n(x)}=[A_1(x), A_2(x), \ldots, A_n(x)].\]
It is well known (see for instance \cite{Fuchs2002}) that
$$\left|x-\frac{P_n(x)}{Q_n(x)}\right|_\infty=\frac{1}{|Q_n(x)|_\infty|Q_{n+1}(x)|_\infty}$$
and $$|Q_n(x)|_\infty=\prod_{i=1}^n|A_i(x)|_\infty=q^{\sum_{i=1}^n\deg A_i(x)}.$$
Both of these facts give information on the relative  error of approximation for $x$ by consecutive convergents as
\[\log_q\frac{|x-P_{n-1}(x)/Q_{n-1}(x)|_{\infty}}{|x-P_{n}(x)/Q_{n}(x)|_{\infty}}=\deg A_{n}(x)+\deg A_{n+1}(x)\]
and
\[\log_q\prod\limits_{i=n}^{n+k}\frac{|x-P_{2i-1}(x)/Q_{2i-1}(x)|_{\infty}}{|x-P_{2i}(x)/Q_{2i}(x)|_{\infty}}=\sum_{i=2n}^{2n+2k+1}\deg A_i(x).\]

\medskip

Thus $\mathcal F_{k}(\Phi)$ describes the set of $x$ which satisfy certain relative growth speed of consecutive approximations by convergents when $k$ is even.

In this paper, we calculate the $\nu$-measure and Hausdorff dimension of the set $\mathcal F_{k}(\Phi)$.
Without loss of generality, we can assume $\Phi(n)\ge k$ since $\sum\limits_{i=1}^k\deg A_{n+i}(x)\ge k$ for any irrational $x\in \mathbb{F}_q((z^{-1}))$ and any $n\ge0$.
\begin{thm}\label{measure-thm}
Let $\Phi: \mathbb{N} \rightarrow [k,\infty)$ be a positive function. Then
\begin{equation*}
\nu(\mathcal F_{k}(\Phi))=
\left\{
             \begin{array}{lll}
             0, & {\rm if} & \sum\limits_{n=1}^{\infty} \frac{\Phi^{k-1}(n)}{q^{\Phi(n)}}< \infty,\\
             1, & {\rm if} & \sum\limits_{n=1}^{\infty} \frac{\Phi^{k-1}(n)}{q^{\Phi(n)}}= \infty.
             \end{array}
\right.
\end{equation*}
\end{thm}

The Hausdorff dimension of $\mathcal F_{k}(\Phi)$ is completely given by the following result.


\begin{thm}\label{main-dimension-thm}
Let $\Phi: \mathbb{N} \rightarrow [k,\infty)$ be a positive function. Let $$ B:=\liminf\limits_{n\rightarrow \infty}\frac{\Phi(n)}{n}, \
\log b:=\liminf\limits_{n\rightarrow \infty}\frac{\log \Phi(n)}{n}.$$
Then
\begin{equation*}
\dim _{\mathcal{H}}\mathcal F_k(\Phi )=\left\{
\begin{array}{lll}
1 &  \mathrm{if} & \ B=0; \\ [3ex]
\frac{1}{1+b} &  \mathrm{if} & \ B=\infty; \\ [3ex]
s_k(B) & \mathrm{if} &  \ 0<B<\infty,
\end{array}
\right.
\end{equation*}
where $s_k(B)$ is the unique solution of the equation
\begin{equation}  \label{def of skB}
\sum\limits_{j=1}^{\infty} \frac{ (q-1)q^j}{q^{2js+Bf_{k}(s)}}=1
\end{equation}
and for any $i\geq 1$,  $f_i(s)$ is given by the following recursive formula,

\begin{equation}\label{F_k define}
f_{1}(s)=s, \ f_{i+1}(s)=\frac{sf_{i}(s)}{1-s+f_{i}(s)} \textmd{ for } \ i\geq 1.
\end{equation}

%
%
%
%

\end{thm}


It is worth noting that the case $B=\infty$ further leads to three sub-cases.
\begin{equation*}
\dim_{\mathcal H} \mathcal F_k(\Phi)=
\left\{
             \begin{array}{lll}
             \frac{1}{2} & {\rm if} & b=1,\\ [1ex]
             \frac{1}{b+1} & {\rm if} & 1<b<\infty,\\ [1ex]
             0 & {\rm if} & b=\infty.
             \end{array}
\right.
\end{equation*}

Another natural problem that we resolve in this paper is to determine the size of the following set which is obtained by replacing ``infinitely many $n$'' in the definition of $\mathcal F_k(\Phi)$ with ``for all $n$''.
Let $$\mathcal{G}_k(\Phi)=\left\{x \in I: \sum_{i=1}^k\deg A_{n+i}(x)\geq \Phi(n) \textmd{ for all } n\geq 0\right\}.$$

We calculate the Hausdorff dimension of $\mathcal{G}_k(\Phi)$ which turns out to be independent of $k$.
\begin{thm}\label{main theorem 2 for all}
Let $\Phi(n)$ be a positive function such that $\Phi(n)\rightarrow \infty$ as $n\rightarrow \infty$. Then
$$\dim_{\mathcal H}\mathcal{G}_k(\Phi)=\frac{1}{a+1},$$
where $1\le a\le \infty$ is defined by $\log a=\limsup\limits_{n\to\infty}\frac{\log\Phi(n)}{n}$.
\end{thm}
When $k=1$, Theorem \ref{main theorem 2 for all} is proved as Theorem 2.3 in \cite{Hu-Wang-Wu-Yu}.

\medskip

The paper is structured in the following way.  In section \ref{prel} we group together basic definitions and some auxiliary results that we refer to in proving our results in subsequent sections. In section \ref{Haarmeasure}, we prove Theorem \ref{measure-thm}.  In section \ref{dimupper} we prove the upper bound and in section \ref{dimlower} the lower bound estimates for Theorem \ref{main-dimension-thm} for case $\Phi(n)=nB$ with $0<B<\infty$ only. In section \ref{conclusion}, we combine all the dimension estimates for all the cases to complete the proof of Theorem \ref{main-dimension-thm}. Finally, we prove Theorem \ref{main theorem 2 for all} in the last section.

\medskip

\noindent {\bf Acknowledgements.}  The first-named author is supported by Natural Science Foundation of China(11701261) and the China Scholarship Council. The second-named author was supported by the Australian Research Council Discovery Project (200100994). We thank the anonymous referee for several helpful comments which has lead to an improved presentation and clarity of proofs.

\section{Preliminaries and auxiliary results}\label{prel}
We first introduce some fundamental properties of continued fractions in the field of formal Laurent series.  
Let \[\frac{P_n(x)}{Q_n(x)}=[A_1(x), A_2(x), \ldots, A_n(x)]\]  be  the $n$th convergents of $x$. The convergents can be obtained  from the following recursive formulae:
\begin{align*}
 P_{-1}(x)&=1,  \quad P_0(x)=0, \quad  P_n(x)=A_n(x)P_{n-1}(x)+P_{n-2}(x), \quad  (n\geq 2), \\
 Q_{-1}(x)&=0, \quad Q_0(x)=1,  \quad Q_n(x)=A_n(x)Q_{n-1}(x)+Q_{n-2}(x), \quad (n\geq 2).
  \end{align*}
We list some useful properties of these convergents, see \cite{Fuchs2002,Nied88} for their proofs.

\begin{pro}[\hspace*{-4px}\cite{Fuchs2002, Nied88}]\label{property of convergents}
Let $x\in {\mathbb{F}((z^{-1}))}$. Then for all $n\ge1$,
\begin{enumerate}
  \item[(i)] $(P_n(x),Q_n(x))=1$.
\smallskip
\item[(ii)]  $Q_n(x)P_{n-1}(x)-P_n(x)Q_{n-1}(x)=(-1)^n$.
\smallskip
\item[(iii)] $|Q_n(x)|_{\infty}=\prod_{i=1}^n |A_i(x)|_{\infty}.$
\smallskip
\item[(iv)] $\left|x-\frac{P_n(x)}{Q_n(x)}\right|_{\infty}=\frac{1}{|Q_n(x)Q_{n+1}(x)|_{\infty}}=\frac{1}{|A_{n+1}(x)Q^2_n(x)|_{\infty}}.$
\end{enumerate}
\end{pro}

We will make a frequent use of the following two important properties of the continued fractions over formal power series. The first is that the Haar measure $\nu$ is preserved by the Gauss map $T$ \cite[Lemma 3]{Nied88}.  The second property is that the $\deg A_i(x)$ form a sequence of independent and identically distributed random variables  \cite[Lemma 4]{Nied88}.

\begin{pro}[\hspace*{-4px}\cite{Nied88}]\label{iid of nju}
  Let $T$ be defined by \eqref{Gauss transformation in formal series}. Then
  \begin{enumerate}
  \item[(i)] The Gauss map $T$ is measure preserving with respect to $\nu$.
\smallskip
\item[(ii)]  $\deg A_1(x), \deg A_2(x)\ldots$ are independent and identically distributed random variables with respect to $\nu$.
\end{enumerate}
\end{pro}
These properties are in contrast to the  real numbers case. For instance,  there exists an interval $B\subset [0, 1)$ such that $T_G^{-1}B$ and $B$ have different Lebesgue measure. However, it should not be confused with the fact that the Gauss measure and Lebesgue measure are equivalent.


\medskip

For any polynomials $A_1,A_2,\ldots,A_n\in \mathbb{F}_q[z]$ of positive degree, we call
\[I(A_1,\ldots,A_n):=\{x\in I: A_1(x)=A_1,\ldots,A_n(x)=A_n\}\]
an $n$th order {\em cylinder}. For any  subset $U\subset I$, its diameter $|U|$ can be defined as
\[|U|=\sup\left\{|x-y|_{\infty}: x,y\in U\right\}.\]
\begin{pro}[\hspace*{-4px}\cite{Nied88}]\label{measure and length of cylinders}
  The cylinder $I(A_1,\dots,A_n)$ is a closed disc with diameter
  \[|I(A_1,\dots,A_n)|=q^{-2\sum_{i=1}^n\deg A_i-1}\]
  and Haar measure
  \[\nu(I(A_1,\dots,A_n))=q^{-2\sum_{i=1}^n\deg A_i}.\]
\end{pro}

%

\begin{lem}\label{cyl length lem-1}
Let $A_1, A_2, \ldots, A_n\in\mathbb{F}_q[z]$ be polynomials of positive degree and $m\ge 1$ be an integer. Let
$$G(A_1, A_2, \ldots, A_n)=\bigcup_{\deg A_{n+1}\geq m} I(A_1, A_2, \ldots, A_{n+1}).$$
Then
\begin{equation*}\label{cyl length equa}
 |G(A_1, A_2, \ldots, A_n)|=q^{-m-2\sum_{i=1}^{n}\deg A_i}.
\end{equation*}
\end{lem}
\begin{proof} Let $x,y\in G(A_1, \ldots, A_n)$ with \[x\in I(A_1,  \ldots,A_n, A_{n+1}), ~ y\in I(A_1,\ldots,A_n, A^*_{n+1})\] for some $A_{n+1},A^*_{n+1}\in \mathbb{F}_q[z]$. Without loss of generality, we assume that $m\leq \deg A_{n+1}\le \deg A^*_{n+1}$. Let $x_1=T^{n+1}(x)$ and
$y_1=T^{n+1}(y)$, where $T$ is defined by \eqref{Gauss transformation in formal series}. Then
\begin{align*}
x&=\frac{(A_{n+1}+x_1)P_n+ P_{n-1}}{(A_{n+1}+x_1)Q_{n}+Q_{n-1}}\\
 y&=\frac{(A^*_{n+1}+y_1)P_n+ P_{n-1}}{(A^*_{n+1}+y_1)Q_{n}+Q_{n-1}}.\end{align*}
 So \[|x-y|_{\infty}=\frac{|(A_{n+1}+x_1-A^*_{n+1}-y_1)(P_{n}Q_{n-1}-P_{n-1}Q_{n})|_{\infty} }{|(A_{n+1}+x_1)Q_{n}+Q_{n-1}|_{\infty} |(A^*_{n+1}+y_1)Q_{n}+Q_{n-1}|_{\infty}}.\]
By Proposition \ref{property of convergents}, it follows that
\[|x-y|_{\infty}=\left|\frac{A_{n+1}-A^*_{n+1}}{A_{n+1}A^*_{n+1}Q^2_n}\right|_\infty\le q^{-m-2\sum_{i=1}^n\deg A_i}\]
since $m\le \deg A_{n+1}\le \deg A^*_{n+1}$. The equality holds in the above inequality when $\deg A_{n+1}=m$ and $\deg (A^*_{n+1}- A_{n+1})=\deg A^*_{n+1}$.
\end{proof}

\begin{lem}\label{measure-lem}
The number of cylinders
$I(A_{1},A_{2},\ldots,A_{k})$ such that $\deg A_{1}+\deg A_{2}+\cdots+\deg A_{k}=m$
is ${m-1\choose {k-1}}(q-1)^{k}q^m$.
\end{lem}
\begin{proof}
  The number of integer vectors $(n_1,n_2,\ldots,n_k)$ such that $n_i\ge1$ for all $1\le i\le k$ and $n_1+n_2+\cdots+n_k=m$ is
  $ {m-1 \choose {k-1}}$. Since the number of polynomials in $\mathbb{F}_q[z]$ of degree $n_i$ is $(q-1)q^{n_i}$, the conclusion follows.
\end{proof}

In the next two lemmas we investigate the recursive formula \eqref{F_k define} that plays a significant role in the Hausdorff dimension estimates. Recall that 
\begin{equation*}
f_{1}(s)=s, \ f_{i+1}(s)=\frac{sf_{i}(s)}{1-s+f_{i}(s)} \textmd{ for } \ i\geq 1.
\end{equation*}

\begin{lem}\label{iterative mono}
  For any given $i\geq 1$, $f_i(s)$ is continuous, strictly monotonically increasing on $[0,1]$, and differentiable in $(0,1)$.
\end{lem}

\begin{proof} Clearly, $f_i(0)=0$ and $f_i(1)=1$. Note also that $f_i(s)$ is always a rational function for all $i\ge1$.
We shall prove the conclusion by induction on $i$. When $i=1$, the conclusion holds for $f_1(s)=s$.
Suppose that the conclusion follows for some $i\ge1$. Then
\[1-s+f_i(s)\ge 1-s+f_i(0)=1-s\ne 0\]
if $s\in[0,1)$ and
\[1-s+f_i(s)=f_i(1)=1\]
if $s=1$. Thus $1-s+f_i(s)\ne 0$ for all $s\in[0,1]$. Since
$$f_{i+1}(s)=\frac{sf_i(s)}{1-s+f_i(s)},$$
it follows that the rational function $f_{i+1}(s)$ is continuous on $[0,1]$ and differential in $(0,1)$. Moreover,
$$f'_{i+1}(s)=\frac{f_i^{2}(s)+f_i(s)+s(1-s)f'_i(s)}{(1-s+f_i(s))^2}$$
for $s\in(0,1)$. Since $f_i(s)>f_i(0)=0$ and $f_i'(s)>0$ for any $s\in(0,1)$, it follows that
$f'_{i+1}(s)>0$ for any $s\in(0,1)$. This completes the proof.
\end{proof}

\begin{lem}\label{explicit-f-k} Let $f_i(s)$ be defined by \eqref{F_k define} for $s\in[0,1]$. Then, 

\begin{equation*}
f_i(s)=\frac{s^i}{\sum_{m=0}^{i-1}s^m(1-s)^{i-1-m}}=
\left\{
             \begin{array}{lll}
           \frac{1}{2i} & \rm{if} & \ s=\frac12,\\ [1ex]
             \frac{s^i(2s-1)}{s^i-(1-s)^i} & \rm{if} & \ s\ne\frac12.
             \end{array}
\right.
\end{equation*}

\end{lem}
\begin{proof} By \eqref{F_k define} and Lemma \ref{iterative mono}, we have $f_i(0)=0$, $f_i(1)=1$ and $0<f_i(s)<1$ for any $i\ge 1$ and $s\in(0,1)$.
Since $f_{i+1}(s)=\frac{sf_i(s)}{1-s+f_i(s)}$, we have
\[\frac{1}{f_{i+1}(s)}=\frac{1-s}{s}\frac{1}{f_{i}(s)}+\frac{1}{s}\]
for $s\in(0,1)$.
Thus if $s=1/2$, we have $$\frac{1}{f_{i}(1/2)}=2+2(i-1)=2i \quad \Longrightarrow  \quad  f_i(1/2)=\frac{1}{2i}.$$
If $s\ne 1/2$, we have
$$\begin{aligned}
 \frac{1}{f_{i+1}(s)}&= \left(\frac{1-s}{s}\right)^i\frac{1}{f_{1}(s)}+\frac{1}{s}\sum_{j=0}^{i-1}\left(\frac{1-s}{s}\right)^j\\
 &=\frac{s^{i+1}-(1-s)^{i+1}}{s^{i+1}(2s-1)}.
\end{aligned}
$$
So we always have
$$f_i(s)=\frac{s^i}{\sum_{m=0}^{i-1}s^m(1-s)^{i-1-m}}.$$
\end{proof}
%
It follows from Lemma \ref{explicit-f-k} that if $1/2<s<1$ then
\[f_i(s)=\frac{2s-1}{1-(\frac{1-s}{s})^i}\]
and \[0<2s-1<f_{i+1}(s)< f_i(s)\le s\] for any $i\geq1$.



\begin{lem}\label{dimension func lem}
  For any given $k\geq 1,$ let $s_k(B)$ be the unique solution in $(1/2,1)$ to the equation
 \begin{equation*}
   \sum\limits_{j=1}^{\infty} (q-1)q^j \frac{1}{q^{2js+Bf_{k}(s)}}=1.
 \end{equation*}
Then $s_k(B)$ is continuous with respect to $B$. Moreover,
$$\lim\limits_{B\rightarrow 0}s_k(B)=1,
\lim\limits_{B\rightarrow \infty}s_k(B)=\frac{1}{2}.$$
\end{lem}

The proof of this Lemma is similar to Lemma $7.1$ in  \cite{Hu-Wang-Wu-Yu}. For completeness, we include a proof here.

\begin{proof}
  (i) Let
  \[h_{B}(s)=\sum\limits_{j=1}^{\infty}(q-1)q^j\frac{1}{q^{2js+Bf_k(s)}}.\]
 Then for $s\in(1/2,1]$, $h_{B}(s)=q^{-Bf_k(s)}\frac{q-1}{q^{2s-1}-1}$ which is monotonically decreasing and continuous.
Moreover,  $h_{B}(1)=q^{-B}<1$ and $h_{B}(s)>1$ when $s\in(1/2,1/2+\delta)$ for some $\delta>0$ small enough. Thus there exists a unique
$s_k(B)\in(1/2,1)$ such that $h_{B}(s)=1$.

(ii) For any $\epsilon>0$, it suffices to prove that
\[|s_k(B)-s_k(B')|<\epsilon\]
for any $|B'-B|<\epsilon$. We first consider the case $B-\epsilon<B'<B$ and prove that
\[s_k(B)<s_k(B')<s_k(B)+\epsilon.\]
Since $s_k(\cdot)$ is monotonically decreasing, the left hand part of the inequality is trivial. Whereas, the estimate
$$
\begin{aligned}
  h_{B'}(s_k(B)+\epsilon)&=\sum\limits_{j=1}^{\infty}(q-1)q^j\frac{1}{q^{2j(s_k(B)+\epsilon)+B'f_k(s_k(B)+\epsilon)}}\\
  &\le q^{-2\epsilon}\sum \limits_{j=1}^{\infty}(q-1)q^j\frac{1}{q^{2js_k(B)+B'f_k(s_k(B)+\epsilon)}}\\
  &\le q^{-2\epsilon}\sum \limits_{j=1}^{\infty}(q-1)q^j\frac{1}{q^{2js_k(B)+B'f_k(s_k(B))}}\\
  &=q^{-2\epsilon} h_B(s_k(B))q^{(B-B')f_k(s_k(B))}\\
  &\le q^{-2\epsilon}q^{B-B'}\le q^{-\epsilon}<1,
\end{aligned}
$$
implies that $s_k(B')<s_k(B)+\epsilon$. Similarly, in the case $B<B'<B+\epsilon$, we also have
\[s_k(B)-\epsilon<s_k(B')<s_k(B).\]

(iii) Since $h_{B}(1)=q^{-B}<1$, we always have $s_k(B)<1$ for $B>0$. Take $s=\frac{2}{2+B}$, where $0<B<2$ such that $s\in(1/2,1)$.
Since $f_k(s)\le s$ for $s\in(1/2,1)$, we have
$$
\begin{aligned}
  h_B(s)&=\sum\limits_{j=1}^{\infty}(q-1)q^j\frac{1}{q^{2js+Bf_k(s)}}\\
  &\ge \sum\limits_{j=1}^{\infty}(q-1)q^j\frac{1}{q^{2js+Bs}}\\
  &\ge \sum\limits_{j=1}^{\infty}(q-1)q^j\frac{1}{q^{2j}}=1.
\end{aligned}
$$
Thus $s_k(B)\ge \frac{2}{2+B}$ when $0<B<2$ and it follows that $\lim\limits_{B\to0}s_k(B)=1$. The proof of the other assertion is similar.

\end{proof}

\subsection{Remarks on Dirichlet improvability}\label{Remark Dirichlet}

The theory of uniform Diophantine approximation concerns improvements to Dirichlet's theorem (1842).  In a recent paper,   Kleinbock and Wadleigh \cite{KlWad16} defined the set of $\phi$-Dirichlet improvable numbers  to be the set of all $x \in \mathbb{R}$ such that
\[|qx-p|< \phi(t), ~ 1\le |q|< t\]
has an integer solution $(p, q)$ for all large enough $t$. Here $\phi$ is a non-increasing function such that $\phi(t) \to 0$ as $t\to \infty$.

We investigate the analogue of Dirichlet improvability over formal Laurent series.

%
%
\begin{pro}\label{pro-1}
 For any $x \in \mathbb{F}_q((z^{-1}))$ and $t> 1$, there exists nonzero $(P,Q)\in \mathbb{F}_q[z]\times \mathbb{F}_q[z]$ such that
\begin{equation*}\label{def of Dir-2-1}
  |Q x-P|_{\infty}\le\frac{1}{t}, ~ ~ |Q|_{\infty}<t.
\end{equation*}
\end{pro}
\begin{proof} Let  $x$ is irrational. Since $t>1$ and $Q_0(x)=1$, there exists $n\ge 0$ such that $|Q_n(x)|_{\infty}< t\le  |Q_{n+1}(x)|_{\infty}$. Then
we have
\[|Q_n(x) x-P_n(x)|_{\infty}=\frac{1}{|Q_{n+1}(x)|_{\infty}}\le\frac{1}{t}.\]
If $x$ is rational, write $x=A/B$ with co-prime polynomials $A$ and $B$. If $t\le |B|_{\infty}$, we can get the conclusion by the same arguments as in the irrational case. If $t> |B|_{\infty}$, we get the conclusion by taking $Q=B$ and $P=A$.
\end{proof}
This proposition is an analogue of Dirichlet's Theorem in the field of formal Laurent series. However, note that it is slightly different from the form of Dirichlet's Theorem over formal Laurent series in \cite[Theorem 1.1]{GangulyGhosh}.

Now we define the
$\phi$-Dirichlet improvable set $\mathcal Dir(\phi)$ in the field formal Laurent series field as follows.
Let $\mathcal Dir(\phi)$  be  the set of all $x \in \mathbb{F}_q((z^{-1}))$ such that
\begin{equation}\label{def of Dir-2}
  |Q x-P|_{\infty}\le \phi(t), ~ ~ |Q|_{\infty}<t
\end{equation}
has a nonzero solution $(P,Q)\in \mathbb{F}_q[z]\times \mathbb{F}_q[z]$ for all large enough $t$.

For $x\in \mathbb{F}_q((z^{-1}))$, define
$$\|x\|=\min\limits_{P\in\mathbb{F}_q[z]} |x-P|_{\infty}.$$

\begin{lem}\label{first equ-2}
   Let $\phi$ be non-increasing. Then an irrational $x\in \mathcal Dir(\phi)$ if and only if
  \begin{equation}\label{equvalent def-2}
   \|Q_{n-1}(x) x\|\le \phi(|Q_n(x)|_{\infty})
  \end{equation}
   for all sufficiently large $n$, where $P_n(x)/Q_n(x)$ is the $n$th convergent of $x$.
\end{lem}
\begin{proof}
 Suppose $x\in \mathcal Dir(\phi)$. Take $t=|Q_n(x)|_{\infty}$ for large enough $n$.
  Then by \eqref{def of Dir-2}, there exists $0\ne Q\in \mathbb{F}_q[z]$ such that
  \[\|Q x\|\le\phi(|Q_n(x)|_{\infty}),~~ |Q|_{\infty}<|Q_n(x)|_{\infty}.\]
  By \cite[Lemma 1]{KimNakada},
$$\|Qx\| \ge\|Q_{n-1}x\| \textmd{ whenever }  |Q|_{\infty}< |Q_{n}(x)|_{\infty}.$$
So we have $\|Q_{n-1}(x) x\|\le\phi(|Q_n(x)|_{\infty})$ for all large enough $n$.

Conversely, suppose $\|Q_{n-1}(x) x\|\le\phi(|Q_n(x)|_{\infty})$ for all $n\ge N$. Then for any $t\ge |Q_N(x)|_{\infty}$,  there exists $n\ge N$ such that $|Q_{n-1}(x)|_{\infty}<t\le |Q_n(x)|_{\infty}$. Since $\phi$ is non-increasing, we have
\[\|Q_{n-1}(x) x\|\le\phi(|Q_n(x)|_{\infty})\le \phi(t).\]
Thus $x$ is $\phi$-Dirichlet.
\end{proof}

Since $$\|Q_{n-1}(x) x\|=|Q_n(x)|_{\infty}^{-1}=q^{-\deg Q_n(x)}=q^{-(\deg A_1(x)+\deg A_2(x)+\cdots+\deg A_n(x))},$$ it follows that
$x$ is $\phi$-Dirichlet improvable if and only if
\begin{equation}\label{conditon of Qn-2}
  |Q_n(x)|_{\infty} \phi(|Q_n(x)|_{\infty})=q^{\sum_{i=1}^n\deg A_i(x)}\phi(q^{\sum_{i=1}^n\deg A_i(x)})\ge1
\end{equation}
for all large enough $n$ by \eqref{equvalent def-2}.
Thus we get the following by Lemma \ref{first equ-2} and estimate \eqref{conditon of Qn-2}.

\begin{lem}\label{equav-2}
    Let $\phi$ be non-increasing. Then an irrational $x\in \mathcal Dir(\phi)$ if and only if
  \begin{equation}\label{equ-def-qn-2}
  \phi(q^{\sum_{i=1}^n\deg A_i(x)})\ge q^{-\sum_{i=1}^n\deg A_i(x)}
  \end{equation}
   for all sufficiently large $n$.
\end{lem}

Clearly, it follows from  Lemma \ref{equav-2} that  $\mathcal Dir(\phi)= \mathbb{F}_q((z^{-1})) $ if $\phi(t)=1/t$.
\begin{cor}
 Let $\phi$ be non-increasing.  If $\phi(q^n)q^n<1$ for infinitely many $n$, then $\mathcal Dir(\phi)\ne \mathbb{F}_q((z^{-1})) $.
\end{cor}

\begin{proof}
 Take $x\in\mathbb{F}_q((z^{-1}))$ with $\deg A_i(x)=1$ for all $i\ge 1$. Then \eqref{equ-def-qn-2} does not hold for infinitely many $n$. It follows that $x\not\in \mathcal Dir(\phi)$.
\end{proof}

\section{Proof of Theorem \ref{measure-thm}}\label{Haarmeasure}
As we have mentioned, the case $k=1$ was proved by Niederreiter \cite{Nied88}. Now we assume $k\ge2$.
For each $n\ge 0$, let \[E_n=\left\{x\in I: \deg A_{n+1}(x)+\deg A_{n+2}(x)+\cdots+\deg A_{n+k}(x)\geq \Phi(n)\right\}\] and
$$F_n=\left\{x\in I: \deg A_{1}(x)+\deg A_{2}(x)+\cdots+\deg A_{k}(x)\geq \Phi(n)\right\}.$$
Then \[\mathcal{F}_k(\Phi)=\left\{x\in I: x\in E_n\textmd{ for infinitely many }n\right\}.\]
Since the Haar measure is $T$-invariant (Proposition \ref{iid of nju} (i)), we have $\nu (E_n)=\nu (F_n)$. Next we calculate $\nu(F_n)$.
Note that for any $m\geq k$, we have
\begin{equation*}
\begin{aligned}
            \left\{x\in I: \sum_{i=1}^k\deg A_{i}(x)=m\right\}              =&\bigcup_{\substack{A_1,\ldots,A_k:\sum_{i=1}^k\deg A_{i}=m }}I(A_{1},A_{2},\ldots,A_{k}) . \\
\end{aligned}
\end{equation*}
By Lemma \ref{measure-lem} and Proposition \ref{measure and length of cylinders}, it follows that
\begin{equation*}
\begin{aligned}
         \nu\left( \left\{x\in I: \sum_{i=1}^k\deg A_{i}(x)=m\right\}  \right) =&{m-1\choose k-1} (q-1)^{k} q^m q^{-2m}\\ =&{m-1\choose k-1} (q-1)^{k} q^{-m} .
\end{aligned}
\end{equation*}
We denote by  $\lceil\xi\rceil$ to be the smallest integer no less than $\xi\in \mathbb{R}$. Then
\begin{equation*}
\begin{aligned}
            \nu(F_n)&= \sum\limits_{m=\lceil \Phi(n)\rceil}^{\infty}     \nu\left( \left\{x\in I: \sum_{i=1}^k\deg A_{i}(x)=m\right\}  \right)  \\
             &=\sum\limits_{m=\lceil \Phi(n)\rceil}^{\infty} {m-1\choose k-1} (q-1)^{k}q^{-m}\\
            &\ge { \lceil\Phi(n)\rceil-1\choose k-1}(q-1)^{k}q^{-\lceil\Phi(n)\rceil}\\
            &= \frac{\lceil\Phi(n)\rceil^{k-1}}{(k-1)!}(q-1)^{k}q^{-\lceil\Phi(n)\rceil}\prod\limits_{i=1}^{k-1}\left(1-\frac{i}{\lceil\Phi(n)\rceil}\right)\\
            &\ge \frac{\Phi^{k-1}(n)}{(k-1)!}(q-1)^{k}q^{-\Phi(n)-1}\left(1-\frac{k-1}{k}\right)^{k-1}\\
            &=c_1\frac{\Phi^{k-1}(n)}{q^{\Phi(n)}}
\end{aligned}
\end{equation*}
where $c_1=k^{-k+1}q^{-1}(q-1)^k/(k-1)!$. Next for the upper bound of $\nu(F_n)$, note that
\[\nu(F_n)\le \sum\limits_{m=\lceil \Phi(n)\rceil}^{\infty} m^{k-1}(q-1)^kq^{-m}.\]
Let \[b_m=m^{k-1}(q-1)^kq^{-m}\]
for $m\ge k$. Since
\[\lim\limits_{n\to\infty}\frac{b_{m+1}}{b_m}=q^{-1}\]
and
\[\frac{b_{m+1}}{b_m}\le \left(1+\frac{1}{k}\right)^{k-1}q^{-1}\]
for all $m\ge k$,
there exists a constant $c_2$ depending on $k$ and $q$ such that
\[\sum\limits_{m=i}^{\infty}b_m\le c_2b_i\]
for all $i\ge k$. Thus
$$
\begin{aligned}
            \nu(F_n)&\le c_2\lceil\Phi(n)\rceil^{k-1}(q-1)^kq^{-\lceil\Phi(n)\rceil}\\
            &\le c_2 (q-1)^k 2^{k-1}\frac{\Phi^{k-1}(n)}{q^{\Phi(n)}}.
\end{aligned}$$
So there exists a constant $c>0$ depending on $k$ and $q$ such that
\[c^{-1} \frac{\Phi^{k-1}(n)}{q^{\Phi(n)}}\le \nu(F_n)\le c\frac{\Phi^{k-1}(n)}{q^{\Phi(n)}}\]
for all $n\ge1$. From the first Borel-Cantelli Lemma, it follows that the $\nu$ measure of $\mathcal F_k(\Phi)$ is zero if the series $\sum_{n}\frac{\Phi^{k-1}(n)}{q^{\Phi(n)}}$ converges. For the divergence case,
since
\[\sum\limits_{n=0}^{\infty}\nu (E_n)=\sum\limits_{j=0}^{k-1}\sum\limits_{i=0}^{\infty}\nu(E_{ik+j}),\]
there exists an integer $0\le j_0\le k-1$ such that
$\sum\limits_{i=0}^{\infty}\nu(E_{ik+j_0})=\infty$.
 By Proposition \ref{iid of nju} (ii), $E_{j_0}$, $E_{k+j_0}$, $E_{2k+j_0},\ldots$ are independent with respect to $\nu$. Thus by the Borel-Cantelli Lemma,
\[\nu \left(\left\{x\in I: x\in E_{ik+j_0}\textmd{ for infinitely many }i\right\}\right)=1.\]
It follows that $\nu(\mathcal{F}_k(\Phi))=1$.

%

\section{The upper bound of $\dim_{\mathcal H}\mathcal F_{k}(\Phi)$ for $\Phi(n)=nB$}\label{dimupper}
In this section we prove the upper bound of Theorem  \ref{main-dimension-thm} for the case $\Phi(n)=nB$ with $0<B<\infty$. Recall that
\begin{equation*}
\mathcal F_k(\Phi)=\left\{x\in I: \sum_{i=1}^k\deg A_{n+i}(x)\ge \Phi(n) \ \mathrm{for \ infinitely \ many \ }n\in\mathbb{N}\right\}.
\end{equation*}
 For any given $0<B<\infty$, if $\Phi(n)=nB$, we denote $\mathcal F_{k}(\Phi)$ by $\mathcal F_k(B)$.
We shall prove the following theorem.
\begin{thm}\label{dimension of EkB}
$\dim_{\mathcal H} \mathcal F_{k}(B)=s_k(B)$, where $s_k(B)$ is defined by \eqref{def of skB}.
\end{thm}

The proof of the theorem splits into two parts: the upper bound and the lower bound. We prove them separately but before that we state the definition of Hausdorff dimension for completeness.   Let $U\subset I$. Then for any $\rho>0$, any finite or countable collection $\{B_i\}$ of subsets of $I$ with diameters $|B_i|\leq\rho$ such that $U\subset \cup_iB_i$ is called a $\rho$-cover of $U$. Let
$$\mathcal H^t(U)=\lim_{\rho\to 0}\left\{\inf\sum_i|B_i|^t\right\},$$
where the infimum is taken over all possible $\rho$-covers $\{B_i\}$ of $U$. The Hausdorff dimension of $U$ is defined by

$$\dim_{\mathcal H} U=\inf\left\{t\geq 0: \mathcal H^t(U)=0\right\}.$$

We first estimate the upper bound of the Hausdorff dimension of $\mathcal F_k(B)$.
\begin{lem}\label{upper bound of dimension of EkB}
  $\dim_{\mathcal H} \mathcal F_{k}(B)\le s_k(B)$.
\end{lem}
\begin{proof}
We prove this result by induction. In the case $k=1$, the result has already been proven as Theorem $7.2$ in \cite{Hu-Wang-Wu-Yu}, that is,  $\dim_{\mathcal H} \mathcal F_{1}(B)=s_1(B)$. Suppose that the conclusion holds for $k$. Then we show that
$\dim_{\mathcal H} \mathcal F_{k+1}(B)\leq s_{k+1}(B)$.
For any $0<\gamma < B,$ let

\begin{equation*}\label{pqc}
 F_{k+1}(\gamma, B)=\left\{ x\in I: \begin{array}{ll}
\sum\limits_{i=1}^k\deg A_{n+i}(x)\leq n\gamma \ \text{and} \\  \deg A_{n+k+1}(x)\geq nB-\sum\limits_{i=1}^k\deg A_{n+i}(x)
\end{array}\text{ for infinitely many  }n\in \mathbb{N}
\right\}.
\end{equation*}
Then
$$\mathcal F_{k+1}(B)\subseteq \mathcal F_{k}(\gamma)\cup F_{k+1}(\gamma, B).$$
So we have
$$\dim_{\mathcal H} \mathcal F_{k+1}(B)\leq \inf\limits_{0<\gamma <  B} \max\left\{
\dim_{\mathcal H} \mathcal F_{k}(\gamma), \dim_{\mathcal H} F_{k+1}(\gamma, B)\right\}.$$
By the induction hypothesis, we have $\dim_{\mathcal H} \mathcal F_{k}(\gamma) \leq s_k(\gamma)$, where $s_k(\gamma)$ is the unique solution of the equation
\begin{equation}  \label{def-of-s-k-gamma}
\sum\limits_{j=1}^{\infty} (q-1)q^j \frac{1}{q^{2js+\gamma f_{k}(s)}}=1.
\end{equation}
Next we shall give an upper bound of $\dim_{\mathcal H} F_{k+1}(\gamma, B)$.
For every $n\geq 1,$ let
\begin{equation*}
 F_{k+1}(\gamma, B, n)=\left\{ x\in I: \begin{array}{l}
\sum\limits_{i=1}^k\deg A_{n+i}(x)\leq n\gamma \ \text{ and} \  \\ \deg A_{n+k+1}(x)\geq nB-\sum\limits_{i=1}^k\deg A_{n+i}(x)
\end{array}
\right\}.
\end{equation*}
%
Then
$$F_{k+1}(\gamma, B)=\bigcap\limits_{N=1}^{\infty} \bigcup\limits_{n=N}^{\infty} F_{k+1}(\gamma, B, n).$$
So for any $N\geq 1,$ $\left\{F_{k+1}(\gamma, B, n):n\geq N\right\}$ is a cover of the set $F_{k+1}(\gamma, B)$.
Note that
\begin{equation*}
\begin{aligned}
           F_{k+1}(\gamma, B, n)\subseteq
            \bigcup_{\substack{\deg A_{i}\geq 1  \\
             i=1,2,\ldots,n}}
           \bigcup_{\substack{ \sum_{i=1}^{k}\deg A_{n+i}\leq n\gamma }}
             G(A_1, A_2, \ldots, A_{n+k}),
\end{aligned}
\end{equation*}
where
\begin{equation*}
           G(A_1, A_2, \ldots, A_{n+k})=
           \bigcup_{\deg A_{n+k+1}\geq nB-\sum_{i=1}^{k}\deg A_{n+i}}
             I(A_1, A_2, \ldots, A_{n+k+1}).
\end{equation*}
The diameter of $G(A_1, A_2, \ldots, A_{n+k})$, by Lemma \ref{cyl length lem-1}, is given by
$$|G(A_1, A_2, \ldots, A_{n+k})|=q^{-\lceil nB\rceil+\sum_{i=1}^{k}\deg A_{n+i}}q^{-2\sum_{i=1}^{n+k}\deg A_{i}}.$$
Thus \[\Lambda_N:=\bigcup_{n\ge N} ~~~\bigcup_{\deg A_{i}, 1\le i\le n}
         ~~  \bigcup_{\sum_{i=1}^{k}\deg A_{n+i}\leq n\gamma}
             G(A_1, A_2, \ldots, A_{n+k})\]
             is a cover of $F_{k+1}(\gamma, B)$ for any $N\ge1$.  For any $1/2<t<1$,
\begin{align*}
            \mathcal{H}^t(F_{k+1}(\gamma, B))&\le\liminf\limits_{N\to\infty}\sum\limits_{G(A_1, A_2, \ldots, A_{n+k})\in\Lambda_N}|G(A_1, A_2, \ldots, A_{n+k})|^t\\
            &\le \liminf\limits_{N\to\infty}\sum_{n\geq N}\sum_{\substack{\deg A_{i}\geq 1  \\
             i=1,2,\ldots,n}}
              \sum_{\substack{ \sum_{i=1}^{k}\deg A_{n+i}\leq n\gamma}}
             |G(A_1, A_2, \ldots, A_{n+k})|^t\\
         &\le \liminf\limits_{N\to\infty}\sum_{n\geq N} q^{-nBt}
            \left(\sum_{\substack{j\ge 1}}(q-1)q^j\frac{1}{q^{2jt}}\right)^{n}
              \sum_{\substack{ \sum_{i=1}^{k}\deg A_{n+i}\leq n\gamma}}
            q^{-t\sum_{i=1}^{k}\deg A_{n+i}} \\
          &=\liminf\limits_{N\to\infty}\sum_{n\geq N}\left(\sum\limits_{j=1}^{\infty}(q-1)q^j \frac{1}{q^{2jt+Bt}}\right)^{n}
            \sum\limits_{j=k}^{\lfloor n\gamma\rfloor}{j-1\choose k-1}(q-1)^k q^j \frac{1}{q^{jt}}.
\end{align*}
For simplicity, denote $s_{k+1}(B)$ by $\tilde{s}$. Since $\tilde{s}$ is the unique solution of the equation
\begin{equation}  \label{infin series s}
\sum\limits_{j=1}^{\infty} (q-1)q^j \frac{1}{q^{2j\tilde{s}+Bf_{k+1}(\tilde{s})}}=1,
\end{equation}
we have  $1/2<\tilde{s}=s_{k+1}(B)<1$.
%
Since
$$\begin{aligned}
               \sum\limits_{j=k}^{\lfloor n\gamma\rfloor}{j-1\choose k-1}(q-1)^k q^j \frac{1}{q^{jt}}
                &\le \sum\limits_{j=k}^{\lfloor n\gamma\rfloor}j^k(q-1)^k q^{(1-t)j} \\
                &\le (n\gamma)^k (q-1)^k\sum\limits_{j=k}^{\lfloor n\gamma\rfloor} q^{(1-t)j}\\
                &\le (n\gamma)^k (q-1)^k\frac{q^{(1-t)(n\gamma+1)}}{q^{1-t}-1}\\ &=cn^kq^{(1-t)n\gamma}
\end{aligned}$$
with $c=\gamma^k(q-1)^kq^{1-t}/(q^{1-t}-1)$,
we have
$$\begin{aligned}
 \mathcal{H}^t(F_{k+1}(\gamma, B))&\le\liminf\limits_{N\to\infty}\sum_{n\geq N}cn^k(\sum\limits_{j=1}^{\infty}(q-1)q^j \frac{1}{q^{2jt+Bt}})^{n}q^{(1-t)n\gamma}\\
&=\liminf\limits_{N\to\infty}\sum_{n\geq N}cn^k(\sum\limits_{j=1}^{\infty}(q-1)q^j \frac{1}{q^{2jt+Bt-(1-t)\gamma}})^{n}.
\end{aligned}$$
Now we take \[\gamma=\frac{B\tilde{s}}{1-\tilde{s}+f_k(\tilde{s})}.\]
Then we have
\begin{equation}\label{tilde-s-1}
 B\tilde{s}-(1-\tilde{s})\gamma=\gamma f_k(\tilde{s})=Bf_{k+1}(\tilde{s}),
\end{equation}
where the second equality follows by \eqref{F_k define}.
For any small $\epsilon >0$, take
\[t=\tilde{s}+\epsilon.\]
Let \[g(s):=\sum\limits_{j=1}^\infty(q-1)q^j\frac{1}{q^{2js+Bs-(1-s)\gamma}}.\]
The function $g(s)$ is strictly monotonically decreasing on $(1/2,\infty)$ and
\[g(\tilde{s})=\sum\limits_{j=1}^\infty(q-1)q^j\frac{1}{q^{2j\tilde{s}+B\tilde{s}-(1-\tilde{s})\gamma}}=\sum\limits_{j=1}^{\infty} (q-1)q^j \frac{1}{q^{2j\tilde{s}+Bf_{k+1}(\tilde{s})}}=1\] by \eqref{infin series s} and \eqref{tilde-s-1}. It follows that
\[ \sum_{n\geq 1}n^k (g(t))^n=\sum_{n\geq 1}n^k (g(\tilde{s}+\epsilon))^n<\infty\]
and hence
\[\mathcal{H}^t(F_{k+1}(\gamma, B))\le \lim_{N\to\infty}c\sum\limits_{n\ge N}n^k (g(\tilde{s}+\epsilon))^n=0\]
since $g(\tilde{s}+\epsilon)<g(\tilde{s})=1$. Therefore, from the definition of Hausdorff dimension, it follows that
\[\dim_{\mathcal H} F_{k+1}(\gamma, B)\le \tilde{s}.\]
On the other hand, by \eqref{infin series s} and \eqref{tilde-s-1}, we have
\[\sum\limits_{j=1}^{\infty} (q-1)q^j \frac{1}{q^{2j\tilde{s}+\gamma f_{k}(\tilde{s})}}=\sum\limits_{j=1}^{\infty} (q-1)q^j \frac{1}{q^{2j\tilde{s}+B f_{k+1}(\tilde{s})}}=1.\]
Then by \eqref{def-of-s-k-gamma}, we have $s_k(\gamma)=\tilde{s}$.
So $$\dim_{\mathcal H} \mathcal F_k(B)\le \max\{s_k(\gamma), \dim F_{k+1}(\gamma, B)\}=\tilde{s}.$$
\end{proof}

\section{The lower bound of $\dim_{\mathcal H}\mathcal F_{k}(\Phi)$ for $\Phi(n)=nB$}\label{dimlower}

In this section we shall prove that $\dim_{\mathcal H} \mathcal F_{k}(B)\geq s_{k}(B)$, where $s_{k}(B)$ is defined by \eqref{def of skB}. To prove the lower bound of the Hausdorff dimension, we shall use the well-known mass distribution principle \cite[Proposition 4.2]{Falconer_book2013}.
\begin{lem}\label{mass distribution priciple}
Let $F\subset I$ and let $\mu$ be a measure with support contained in $F$. If there are positive constants $c,\delta$ such that
$$\mu(U)<c |U|^s$$
for all discs with $|U|\le \delta$, we have
\[\dim_{\mathcal H} F\ge s.\]
\end{lem}

 To apply this Lemma  we will construct a suitable Cantor like subset of $ \mathcal F_{k}(B)$ that supports the $\mu$-measure. We then distribute $\mu$-measure on basic subsets. Finally, we calculate measure for any disc satisfying the hypothesis of Lemma \ref{mass distribution priciple} to conclude the proof.

For $k=1$, we  already have  $\dim_{\mathcal H} \mathcal F_{1}(B)$ from \cite{Hu-Wang-Wu-Yu}. So we
assume $k\geq 2$ throughout the rest of the section.

For $M\ge 2$, we denote by $s_{k,M}(B)$  the unique solution of
 \begin{equation}  \label{dimension func form}
   \sum\limits_{j=1}^{M} (q-1)q^j \frac{1}{q^{2js+Bf_{k}(s)}}=1.
 \end{equation}
 Then
  \[\lim\limits_{M\to\infty}s_{k,M}(B)=s_{k}(B).\]
For the remainder of this section, we write $s=s_{k,M}(B)$ for simplicity. Define parameters $\alpha_i,$ for $1\le i\le k $ satisfying
\begin{equation}\label{definition of alpha-i}
  s\alpha_i=(1-s)\alpha_{i-1}, \sum\limits_{i=1}^k\alpha_i=B.
\end{equation}
Rewriting this relation,  we have
\[\alpha_i=\frac{1-s}{s}\alpha_{i-1},\]
and \[\alpha_1\sum\limits_{i=0}^{k-1}\left(\frac{1-s}{s}\right)^i=B \quad \Longrightarrow \quad \alpha_1=\frac B{\sum\limits_{i=0}^{k-1}\left(\frac{1-s}{s}\right)^i}=\frac {Bs^k(2s-1)}{s(s^k-(1-s)^k)}.\]
Hence
\[\alpha_i= \frac{s^{k-i}(2s-1)(1-s)^{i-1}}{s^{k}-(1-s)^{k}}B, ~1\le i\le k.\]
Finally, we identify (following Lemma \ref{explicit-f-k}) that
\begin{equation}\label{f-k-alpha}
 Bf_k(s)=s\alpha_1.
\end{equation}

\subsection{A subset  of $\mathcal F_k(B)$}\label{A subset  of F_k(B)}
Let $M\ge2$ be an integer and $\epsilon>0$ a real number small enough. Let $\{n_j\}$ be a subsequence of positive integers  satisfying
\begin{equation}\label{def of nj}
\min_{1\le i\le k}\frac{n_j\alpha_i}{n_j\alpha_i+2}\ge\frac{s-\epsilon}{s},~~~ \frac{n_{j+1}-n_j-k}{n_{j+1}}\ge \frac{s-\epsilon}{s}
\end{equation}
for all $j\ge1$, where $\alpha_1,\ldots,\alpha_k$ are defined by \eqref{definition of alpha-i}. Now we shall
construct a subset $E\left(B,M,\epsilon,\{n_j\}\right)$ of $\mathcal F_k(B)$ as follows.

\begin{equation*}
E\left(B,M,\epsilon,\{n_j\}\right)=\left\{ x\in I: \begin{array}{ll}
\deg A_{n_j+i}(x)=\lfloor n_j\alpha_i \rfloor +1 ~ \text{ for all } j\ge1,  1\leq i\leq k,\\ [1ex]
1\leq \deg A_{n}(x)\leq M \ \text{ for other } n \end{array}
\right\}.
\end{equation*}
By \eqref{definition of alpha-i}, we have
\[E\left(B,M,\epsilon,\{n_j\}\right) \subseteq \mathcal F_k(B).\]
For the remainder of the section, by using the mass distribution principle we prove the following
\[\dim_{\mathcal H} E\left(B,M,\epsilon,\{n_j\}\right)\ge s-\epsilon.\]

\subsection{Fractal structure of $E\left(B,M,\epsilon,\{n_j\}\right)$}

For any $n\geq 1$, denote by $D_n$ the set of all $(A_1, \ldots, A_n) \in \mathbb{F}_q[z]^n$ such that

\begin{equation*}
\left\{ \begin{array}{ll}
\deg A_{n_j+i}=\lfloor n_j \alpha_i \rfloor +1 \ \text{for any} \
j\geq 1 \ \text{and} \  1\leq i\leq k\  \text{with} \ 1\leq n_j+i \leq n, \\ \\
1\leq \deg A_{m}\leq M \ \text{ for \ other } \ 1\leq m \leq n \end{array}
\right\}.
\end{equation*}
%
Let $$D=\bigcup\limits_{n=1}^{\infty}D_n.$$
For any $n \geq 1$ and $(A_1,\ldots, A_n) \in D_n$, define
$$
J(A_1,\ldots, A_n)=\bigcup\limits_{A_{n+1}:(A_1,\ldots, A_{n+1}) \in D_{n+1}}I(A_1,\ldots, A_{n+1})
$$
and we call $J(A_1,\ldots, A_n)$ a {\em basic set} of order $n$. Note that $J(A_1,\ldots, A_n)$ is a union of finitely many disjoint discs.
Then
\begin{equation*}
E\left(B,M,\epsilon,\{n_j\}\right)=\bigcap\limits_{n \geq 1}\bigcup\limits_{(A_1,\ldots, A_{n}) \in D_{n}}J(A_1,\ldots, A_n).
\end{equation*}

\begin{lem}\label{cyl length lem}
For any $n \geq 1$ and $(A_1,\ldots, A_n) \in D_n$, we have

\begin{equation*}
|J(A_1,\ldots, A_n)|=
\left\{
             \begin{array}{ll}
          q^{-2\sum_{m=1}^n\deg A_m-\lfloor n_j\alpha_i\rfloor-1}, & {\rm if} \ n=n_j+i-1, j\ge1,  1\le i\le k,\\ [2ex]
                   q^{-2\sum_{m=1}^n\deg A_m-1}, &  {\rm otherwise}.
             \end{array}
\right.
\end{equation*}
\end{lem}

\begin{proof}
  Let $x,y\in J(A_1,\ldots, A_n)$ with $x\in I(A_1,\ldots, A_n,A_{n+1})$ and $y\in I(A_1,\ldots, A_n,A^*_{n+1})$. Then as in the proof of Lemma \ref{cyl length lem-1}, we have
  \[|x-y|_{\infty}=\left|\frac{A_{n+1}-A^*_{n+1}}{A_{n+1}A^*_{n+1}Q^2_n}\right|_\infty.\]
  If $n=n_j+i-1$ for some $j\ge1$ and $1\le i\le k$, then $\deg A_{n+1}=\deg A^*_{n+1}=\lfloor n_j\alpha_i\rfloor+1$. So
  \[|x-y|_{\infty}\le q^{-\lfloor n_j\alpha_i\rfloor-1-2\sum_{i=1}^n\deg A_i},\] where the equality holds when $\deg (A_{n+1}- A^*_{n+1})=\lfloor n_j\alpha_i\rfloor+1$.
  If $n\not\in\{n_j+i-1: j\ge1, 1\le i\le k\}$, we have
  \[|x-y|_{\infty}\le q^{-1-2\sum_{i=1}^n\deg A_i}.\] Here the equality holds when $\deg A_{n+1}=1$ and $\deg A^*_{n+1}=M$.

\end{proof}

\subsection{The $\mu$ measure on $E\left(B,M,\epsilon,\{n_j\}\right)$}
To define a measure $\mu$ on $E\left(B,M,\epsilon,\{n_j\}\right)$, we first distribute the mass on basic sets.
\begin{itemize}

\item If $n=1$, define
$$\mu(J(A_1))=q^{-2s \deg A_1 } q^{-Bf_k(s)}.$$

\item If  $2\le  n \leq n_1$,
$$\mu(J(A_1,\ldots, A_n))=q^{-2s \deg A_n } q^{-Bf_k(s)}\mu(J(A_1,\ldots, A_{n-1})).$$

\item  If $n_j+1\le n\le n_j+k$ for some $j\ge1$, write $n=n_j +i$ for some $1 \leq i \le k$. Define
$$\mu(J(A_1,\ldots, A_n))=\frac{1}{(q-1)q^{\lfloor n_j \alpha_i \rfloor+1}}\mu(J(A_1,\ldots, A_{n-1})).$$
It means that the measure is uniformly distributed on the basic sets of order $n$ contained in $J(A_1,\ldots, A_{n-1})$ if $n_j+1\le n\le n_j+k$.

\item If  $n_j +k<n \leq n_{j+1}$, define
$$\mu(J(A_1,\ldots, A_n))=q^{-2s \deg A_n } q^{-Bf_k(s)}\mu(J(A_1,\ldots, A_{n-1})).$$
\end{itemize}

By \eqref{dimension func form}, we have $\mu(I)=1$ and
\[\sum_{A_n:(A_1, A_2, \ldots, A_n)\in D_n }\mu(J(A_1,A_2,\ldots,A_n))=\mu(J(A_1,A_2,\ldots,A_{n-1}))\] for all $n\ge1$. So $\mu$ is well defined on basic sets. Thus it can be extended into a probability measure on $E\left(B,M,\epsilon,\{n_j\}\right)$.  From the definition of $\mu$, we have
\[\mu(I(A_1,A_2,\ldots,A_n))=\mu(J(A_1,A_2,\ldots,A_n))\]
for any $n\ge 1$ and $(A_1,A_2,\ldots,A_n)\in D_n$, since $I(A_1,A_2,\ldots,A_n)$ contains the basic set $J(A_1,A_2,\ldots,A_n)$ but does not intersect any other basic sets of order $n$.

\subsection{The measure on basic sets}
In this section, we prove that
\begin{equation}\label{measure on basic sets}
 \mu(J(A_1,\ldots,A_n))\le|J(A_1,\ldots,A_n)|^{s-\epsilon}
\end{equation}
for any $n\ge n_1$ and $(A_1,\ldots,A_n)\in D_n$.

\medskip
 We consider all the cases one by one.

 \begin{itemize}

 \item $n=n_1$ case.

$$\begin{aligned}
\mu(J(A_1,\ldots, A_{n_1}))&=\prod\limits_{i=1}^{n_1}(q^{-2s \deg A_i } q^{-Bf_k(s)})\\
&=q^{-2s\sum_{i=1}^{n_1}\deg A_i-n_1Bf_k(s)}\\
& =q^{-s(2\sum_{i=1}^{n_1}\deg A_i+n_1\alpha_1)}\\
&\le q^{-(s-\epsilon)(2\sum_{i=1}^{n_1}\deg A_i+n_1\alpha_1+1)}\\
&\le |J(A_1,\ldots,A_{n_1})|^{s-\epsilon},
\end{aligned}$$
where we have used  \eqref{f-k-alpha} in the third equality, \eqref{def of nj} in the first inequality and Lemma \ref{cyl length lem} in the last inequality.

\item
It suffices to prove that if \eqref{measure on basic sets} holds for $n=n_j$, then \eqref{measure on basic sets} also  holds for all $n_j<n\le n_{j+1}$.

\item
If $n=n_j+p$ for some $1\le p\le k-1$, we have
$$\begin{aligned}
&\mu(J(A_1,A_2\ldots, A_{n_j+p}))=\mu(J(A_1,\ldots,A_{n_j}))\prod\limits_{i=1}^{p}\frac{1}{(q-1)q^{\lfloor n_j \alpha_i\rfloor+1 }}\\
&\le \mu(J(A_1,\ldots,A_{n_j}))\prod\limits_{i=1}^{p}\frac{1}{q^{n_j\alpha_i}}= \mu(J(A_1,\ldots,A_{n_j}))\prod\limits_{i=1}^{p}\frac{1}{q^{s n_j(\alpha_i+\alpha_{i+1})}}\\
&=\mu(J(A_1,\ldots,A_{n_j}))q^{-sn_j(\alpha_1+ 2\sum_{i=2}^p\alpha_i+\alpha_{p+1})}\\
&\le  |J(A_1,\ldots,A_{n_j})|^{s-\epsilon} q^{-sn_j(\alpha_1+2\sum_{i=2}^p\alpha_i+\alpha_{p+1})}\\
&= q^{-(s-\epsilon) (2\sum_{i=1}^{n_j}\deg A_i+\lfloor n_j\alpha_1\rfloor+1)}q^{-sn_j(\alpha_1+2\sum_{i=2}^p\alpha_i+\alpha_{p+1})}\\
&\le q^{-2(s-\epsilon)\sum_{i=1}^{n_j+p}\deg A_i-(s-\epsilon)(n_j\alpha_{p+1}+1)}\\
&\le |J(A_1,\ldots,A_{n_j+p})|^{s-\epsilon},
\end{aligned}$$
where we have used \eqref{definition of alpha-i} in the second equality, the fact $\deg A_{n_j+l}=\lfloor n_j\alpha_l\rfloor+1$ for $1\le l\le k$ and \eqref{def of nj} in the third inequality,  Lemma \ref{cyl length lem} in the last inequality.

\item
If $n=n_j+k$, then
$$\begin{aligned}
&\mu(J(A_1,A_2\ldots, A_{n_j+k}))=\mu(J(A_1,\ldots,A_{n_j+k-1}))\frac{1}{(q-1)q^{\lfloor n_j\alpha_k\rfloor+1}}\\
&\le q^{-n_j\alpha_k}\mu(J(A_1,\ldots,A_{n_j+k-1}))\le q^{-n_j\alpha_k}|J(A_1,\ldots,A_{n_j+k-1})|^{s-\epsilon} \\
&= q^{-n_j\alpha_k} q^{-(s-\epsilon)(2\sum_{i=1}^{n_j+k-1}\deg A_i+\deg A_{n_j+k})}\\
&\le q^{-(s-\epsilon)(n_j\alpha_k+2)} q^{-(s-\epsilon)(2\sum_{i=1}^{n_j+k-1}\deg A_i+\deg A_{n_j+k})}\\
&\le q^{-(s-\epsilon)(2\sum_{i=1}^{n_j+k}\deg A_i+1)}\\
&\le |J(A_1,\ldots,A_{n_j+k})|^{s-\epsilon},
\end{aligned}$$
where we have used Lemma \ref{cyl length lem} in the second equality and the last inequality,  and \eqref{def of nj} in the third inequality.

\item
If $n_j+k+1\le n\le n_{j+1}-1$, then
$$\begin{aligned}
\mu(J(A_1,\ldots, A_{n}))&=\mu(J(A_1,\ldots, A_{n_j+k})) \prod\limits_{i=n_j+k+1}^{n}q^{-2s \deg A_i -Bf_k(s)}\\
&\le \mu(J(A_1,\ldots, A_{n_j+k})) q^{-2s\sum_{i=n_j+k+1}^n \deg A_i}\\
&\le |J(A_1,\ldots, A_{n_j+k})|^{s-\epsilon}q^{-2s\sum_{i=n_j+k+1}^n \deg A_i}\\
&=q^{-(s-\epsilon)(2\sum_{i=1}^{n_j+k}\deg A_i+1)} q^{-2s\sum_{i=n_j+k+1}^n \deg A_i}\\
&\le q^{-(s-\epsilon)(2\sum_{i=1}^{n}\deg A_i+1)}\\
&= |J(A_1,\ldots,A_{n})|^{s-\epsilon},
\end{aligned}$$
where we have used Lemma \ref{cyl length lem} in the last two equalities.

\item If $n=n_{j+1}$, we have
$$\begin{aligned}
\mu(J(A_1,\ldots, A_{n_{j+1}}))
&=\mu(J(A_1,\ldots, A_{n_j+k})) \prod\limits_{i=n_j+k+1}^{n_{j+1}}q^{-2s \deg A_i -Bf_k(s)}\\
&= \mu(J(A_1,\ldots, A_{n_j+k}))q^{-2s\sum_{i=n_j+k+1}^n \deg A_i-Bf_k(s)(n_{j+1}-n_j-k)}\\
&\le |J(A_1,\ldots, A_{n_j+k})|^{s-\epsilon}q^{-2s\sum_{i=n_j+k+1}^n \deg A_i-Bf_k(s)(n_{j+1}-n_j-k)}\\
&=q^{-(s-\epsilon)(2\sum_{i=1}^{n_j+k}\deg A_i+1)} q^{-2s\sum_{i=n_j+k+1}^n \deg A_i-\alpha_1s(n_{j+1}-n_j-k) }\\
&\le q^{-(s-\epsilon)(2\sum_{i=1}^{n_j+k}\deg A_i+1)} q^{-2s\sum_{i=n_j+k+1}^n \deg A_i-\alpha_1(s-\epsilon)n_{j+1}}\\
&\le q^{-(s-\epsilon)(2\sum_{i=1}^{n}\deg A_i+\alpha_1n_{j+1}+1)}\\
&\le|J(A_1,\ldots,A_{n_{j+1}})|^{s-\epsilon},
\end{aligned}$$
where we have used Lemma \ref{cyl length lem} in the last equality and the last inequality, \eqref{f-k-alpha} in the last equality,  and \eqref{def of nj} in the second inequality.

\end{itemize}
Thus
\eqref{measure on basic sets} holds for  all basic sets with order $n\ge n_1$.

\subsection{Measure on any disc}
Let $x\in E(B,M,\epsilon,\{n_j\})$ and
$$r_0:=\min\limits_{(A_1,A_2,\ldots,A_{n_1})\in D_{n_1}}|J(A_1,\ldots,A_{n_1})|.$$
 Then there exist polynomials
$A_1,A_2,\ldots$ such that $(A_1,A_2,\ldots,A_{i})\in D_{i}$ and
$x\in J(A_1,\ldots,A_{i})$ for any $i\ge1$.
Let $B(x,r)$ be the disc with the centre $x$ and radius $r$.
 Now we prove that there exists a constant $c$ such that for any $r<r_0$,
\begin{equation}\label{measure on ball x r}
\mu(B(x,r))\le cr^{s-\epsilon}.
\end{equation}
By the definition of $r_0$, there exists a unique $n\ge n_1$ such that
\[|J(A_1,A_2,\ldots,A_{n+1})|\le r<|J(A_1,A_2,\ldots,A_{n})|.\]
Since $r< |J(A_1,\ldots,A_n)|\le |I(A_1,\ldots,A_n)|$ and $I(A_1,\ldots,A_n)$ is also a disc, we have
\begin{equation}\label{disc-contains-cylinder}
 B(x,r)\subset I(A_1,\ldots,A_n).
\end{equation}
We divide the proof into three cases.

\medskip

\noindent {\bf Case I.} We consider $n_j\le n\le n_j+k-1$ for some $j\ge1$.  Write $n=n_j+i-1$ for some $1\le i\le k$. If $B(x,r)$ intersects only one basic set of order $n+1$, then
\[\mu(B(x,r))\le\mu(J(A_1,A_2,\ldots,A_{n+1}))\le |J(A_1,A_2,\ldots,A_{n+1})|^{s-\epsilon}\le r^{s-\epsilon}. \]
Now we consider the case that $B(x,r)$ intersects at least two basic sets of order $n+1$.
By \eqref{disc-contains-cylinder}, the basic sets of order $n+1$ which intersect $B(x,r)$ are all subsets of $J(A_1,\ldots,A_n)$.
So there exists $\tilde {A}_{n+1}\ne A_{n+1}$ such that
\begin{equation}\label{intersect-1}
 B(x,r)\cap J(A_1,\ldots,A_n,\tilde {A}_{n+1})\ne \emptyset.
\end{equation}
Note that
\begin{equation}\label{disjoint-1}
  I(A_1,\ldots,A_n,A_{n+1})\cap J(A_1,\ldots,A_n,\tilde{A}_{n+1})=\emptyset.
\end{equation}
Since $|\cdot|_{\infty}$ is non-Archimedean, any two discs are either disjoint or one is contained in the other.
Both discs $B(x,r)$ and  $I(A_1,A_2,\ldots,A_{n+1})$ intersect $J(A_1,A_2,\ldots,A_{n+1})$, so we have
\[I(A_1,A_2,\ldots,A_{n+1})\subset B(x,r)\]
by \eqref{intersect-1} and \eqref{disjoint-1}. Thus
\begin{equation}\label{length-disc-cylinder}
r\ge |I(A_1,A_2,\ldots,A_{n+1})|.
\end{equation}
For any $(A_1,\ldots,A_n,A^*_{n+1})\in D_{n+1}$, we have
\begin{equation}\label{same-measure-cylinders}
 \nu(I(A_1,\ldots,A_{n+1}))=\nu(I(A_1,\ldots,A_n,A^*_{n+1}))=q^{-2\left(\sum_{m=1}^n\deg A_m+\lfloor n_j\alpha_i\rfloor+1\right)}
\end{equation}
and
\begin{equation}\label{same-length-cylinders}
  |I(A_1,A_2,\ldots,A_{n+1})|=|I(A_1,A_2,\ldots,A_n,A^*_{n+1})|
\end{equation}
since $\deg A_{n+1}=\deg A^*_{n+1}=\lfloor n_j\alpha_i\rfloor+1$.
Thus if $$B(x,r)\cap I(A_1,A_2,\ldots,A_n,A^*_{n+1})\ne \emptyset,$$ we have
\begin{equation}\label{intersect-implies-contained}
 I(A_1,A_2,\ldots,A_n,A^*_{n+1})\subset B(x,r)
\end{equation}
by \eqref{length-disc-cylinder} and \eqref{same-length-cylinders}.
Let $N_r$ be the number of basic sets of order $n+1$ which intersect $B(x,r)$.
Combining \eqref{same-measure-cylinders} and \eqref{intersect-implies-contained}, we get
\[N_r q^{-2\left(\sum_{m=1}^n\deg A_m+\lfloor n_j\alpha_i\rfloor+1\right)} \le\nu(B(x,r))\le qr.\]
In the second inequality, we use the fact that
\[\nu(B(x,r))=\nu(B(x,q^{-m-1}))=q^{-m}\]
if $q^{-m-1}\le r<q^{-m}$ for some $m\ge 1$.
Thus
\[N_r\le q^3 rq^{2\left(\sum_{m=1}^n\deg A_m+ n_j\alpha_i\right)}\]
and
$$\begin{aligned}
  \mu(B(x,r))&\le \sum\limits_{A^*_{n+1}:J\left(A_1,\ldots,A_n,A^*_{n+1}\right)\cap B(x,r)\ne\emptyset} \mu (J(A_1,\ldots,A_n,A^*_{n+1}))\\
  &\le N_r \mu(J(A_1,\ldots,A_n,A_{n+1}))\\
  &=N_r\frac{1}{(q-1)q^{\lfloor n_j\alpha_i\rfloor+1}}\mu(J(A_1,\ldots,A_n))\\
  &\le q^3 rq^{2\sum_{m=1}^n\deg A_m+ n_j\alpha_i}\mu(J(A_1,\ldots,A_n)).
\end{aligned}$$
From \eqref{disc-contains-cylinder}, it follows that \[\mu(B(x,r))\le \mu(I(A_1,\ldots,A_n))=\mu(J(A_1,\ldots,A_n)).\]
Thus
$$\begin{aligned}
  \mu(B(x,r))&\le \min\{q^3 rq^{2\sum_{m=1}^n\deg A_m+ n_j\alpha_i},1\}\mu(J(A_1,\ldots,A_n))\\
  &\le (q^3 rq^{2\sum_{m=1}^n\deg A_m+ n_j\alpha_i})^{s-\epsilon}\mu(J(A_1,\ldots,A_n))\\
  &\le (q^3 rq^{2\sum_{m=1}^n\deg A_m+ n_j\alpha_i})^{s-\epsilon}|J(A_1,\ldots,A_n)|^{s-\epsilon}\\
 &\le (q^3 rq^{2\sum_{m=1}^n\deg A_m+ n_j\alpha_i})^{s-\epsilon} q^{-(s-\epsilon)(2\sum_{m=1}^n\deg A_m+n_j\alpha_i)}\\
 &=(q^3r)^{s-\epsilon}\le q^3r^{s-\epsilon}.
\end{aligned}$$
In the second inequality, we use the fact $\min\{x,y\}\le x^ty^{1-t}$ for $x,y>0$ and $0<t<1$.

\medskip

\noindent {\bf Case II.} We consider $n_j+k\le n\le n_{j+1}-2$ for some $j\ge1$. Since $B(x,r)\subset I(A_1,\ldots,A_n)$, we have
\begin{align*}
\mu(B(x,r))&\le \mu(I(A_1,\ldots,A_n))\\ &=\mu(J(A_1,\ldots,A_n))\\ &\le |J(A_1,\ldots,A_n)|^{s-\epsilon}.\end{align*}
By Lemma \ref{cyl length lem}, we have
\[|J(A_1,\ldots,A_n)|=q^{-2\sum_{i=1}^n\deg A_i-1}\]
and
\[|J(A_1,\ldots,A_n,A_{n+1})|=q^{-2\sum_{i=1}^{n+1}\deg A_i-1}.\]
Since $\deg A_{n+1}\le M$, we have
\[|J(A_1,\ldots,A_n)|\le q^{2M}|J(A_1,\ldots,A_n,A_{n+1})| \]
and hence
\[\mu(B(x,r))\le q^{2M}|J(A_1,\ldots,A_n,A_{n+1})|^{s-\epsilon}\le q^{2M} r^{s-\epsilon}.\]

\medskip

\noindent {\bf Case III.} In this last case, we consider $n=n_{j+1}-1$ for some $j\ge1$. If $B(x,r)$ intersects only one basic set of order $n+1$, we have
\begin{align*}\mu(B(x,r))&\le\mu(J(A_1,A_2,\ldots,A_{n+1}))\\ &\le |J(A_1,A_2,\ldots,A_{n+1})|^{s-\epsilon}\\ &\le r^{s-\epsilon}. \end{align*}
If $B(x,r)$ intersects at least two basic sets of order $n+1$, we have
\[r\ge |I(A_1,\ldots,A_{n+1})|\]
by the same arguments as in the proof of Case I.
Since $n=n_{j+1}-1$, we have $1\le \deg A_{n+1}=\deg A_{n_{j+1}}\le M$. It follows that
\begin{align*}
r&\ge |I(A_1,\ldots,A_{n+1})|\\ &=q^{-2\sum\limits_{i=1}^{n+1}\deg A_i-1} \\ &\ge q^{-2M}q^{-2\sum\limits_{i=1}^{n}\deg A_i-1}\\
&=q^{-2M}|I(A_1,\ldots,A_{n})|\\ &\ge q^{-2M}|J(A_1,\ldots,A_{n})|.
\end{align*}
Thus by \eqref{disc-contains-cylinder}

\begin{align*}
\mu(B(x,r))&\le \mu(I(A_1,\ldots,A_{n}))\\ &=\mu(J(A_1,\ldots,A_{n}))
\\ &\le |J(A_1,\ldots,A_{n})|^{s-\epsilon}\\ &\le q^{2M}r^{s-\epsilon}.
\end{align*}
Thus the inequlity \eqref{measure on ball x r} always holds. By Lemma \ref{mass distribution priciple}, we have
\begin{equation}\label{lower E(B,M,mj)}
  \dim_{\mathcal H} E\left(B,M,\epsilon,\{n_j\}\right)\ge s_{k,M}(B)-\epsilon.
\end{equation}

\begin{proof}[Proof of Theorem \ref{dimension of EkB}]
Since
$E\left(B,M,\epsilon,\{n_j\}\right)\subset \mathcal F_k(B)$ and $\dim_{\mathcal H} E\left(B,M,\epsilon,\{n_j\}\right)\ge s_{k,M}(B)-\epsilon$,
we have
\[\dim_{\mathcal H} \mathcal F_k(B)\ge s_{k,M}(B)-\epsilon.\]
Let $\epsilon\to 0$ and $M\to \infty$. It follows that
\[\dim_{\mathcal H} \mathcal F_k(B)\ge s_{k}(B)\]
and hence, by combining it with the upper bound estimate (Lemma \ref{upper bound of dimension of EkB}), we have
\[\dim_{\mathcal H} \mathcal F_k(B)= s_{k}(B).\]
\end{proof}

\section{Proof of Theorem \ref{main-dimension-thm}}\label{conclusion}
When $k=1$, the conclusion follows from \cite[Theorem 2.4]{Hu-Wang-Wu-Yu}. Therefore, we assume $k\ge 2$ and split the proof into three cases.

\medskip

\noindent {\bf Case 1.} We first handle the case $B=0$. Then
 \[\mathcal F_k(\Phi)\supseteq \mathcal F_1(\Phi).\]
The conclusion follows since $\dim_{\mathcal H} \mathcal F_1(\Phi)=1$.

\medskip

\noindent {\bf Case 2.} Let  $0<B<\infty$. Then for any $\epsilon>0$,
 we have
 \[\mathcal F_k(\Phi)\subseteq \mathcal F_k(B-\epsilon)=\left\{x\in I: \sum_{i=1}^k\deg A_{n+i}(x)\ge (B-\epsilon)n \ \text{for infinitely many }  n\in\N\right\}.\]
 Thus $\dim_{\mathcal H}\mathcal F_k(\Phi)\le s_k(B-\epsilon)$. Letting $\epsilon\to 0$, we have $\dim_{\mathcal H}\mathcal F_k(\Phi)\le s_k(B)$. Now we prove the reverse inequality.
 For any $\delta>0$, $M\ge 2$ and $\epsilon>0$, there exists a sequence $\{n_j\}$ such that
\[\lim\limits_{j\to\infty} \frac{\Phi(n_j)}{n_j}=\liminf\limits_{n\to\infty}\frac{\Phi(n)}{n}=B\]
 and \eqref{def of nj} holds with  $s=s_{k,M}(B+\delta)$, where
 $\alpha_i$ is defined by
 \[ s\alpha_i=(1-s)\alpha_{i-1}, \sum\limits_{i=1}^k\alpha_i=B+\delta.\]
 We define $E(B+\delta,M,\epsilon, \{n_j\})$ as in Section \ref{A subset  of F_k(B)} by replacing $B$ with $B+\delta$. Then
 \[E\left(B+\delta,M,\epsilon, \{n_j\}\right)\subset \mathcal F_k(\Phi)\]
 and
 \[\dim_{\mathcal H} E(B+\delta,M,\epsilon, \{n_j\})\ge s_{k,M}(B+\delta)-\epsilon\]
 by \eqref{lower E(B,M,mj)}.
 It follows that
$$\dim_{\mathcal H} \mathcal F_k(\Phi)\ge s_{k,M}(B+\delta)-\epsilon.$$
 Letting $\epsilon\to 0$, $M\to\infty$ and $\delta\to0$, we get
$\dim_{\mathcal H} \mathcal F_k(\Phi)\ge s_{k}(B)$ by Lemma \ref{dimension func lem}. So the conclusion follows in this case.

\noindent {\bf Case 3.} For the final case, let  $B=\infty$. If $ \deg A_{n+1}(x)+\ldots+\deg A_{n+k}(x)\ge \Phi(n)$ then there exists some $1\le i\le k$ such that
\[\deg A_{n+i}(x)\ge \frac{1}{k}\Phi(n).\]
 Thus
\[ \mathcal F_1(\Phi)\subseteq \mathcal F_k(\Phi)\subseteq\bigcup\limits_{i=1}^k \mathcal{F}_1(\psi_i),\]
 where
  \[\psi_i(n)=\frac{1}{k}\Phi(n+1-i)\]
 for all $1\le i\le k$. Since
 \begin{align*}
 \liminf\limits_{n\to\infty} \frac{\psi_i(n)}{n}&=\liminf\limits_{n\to\infty}\frac{\Phi(n)}{n}=\infty  \ \text{and} \\
  \liminf\limits_{n\to\infty} \frac{\log\psi_i(n)}{n}&=\liminf\limits_{n\to\infty}\frac{\log\Phi(n)}{n}\end{align*}
 for all $1\le i\le k$,
 we have $\dim_{\mathcal H} \mathcal F_1(\psi_i)=\dim_{\mathcal H} \mathcal F_1(\Phi)$. Then the conclusion follows from \cite[Theorem 2.4]{Hu-Wang-Wu-Yu}.
%

\section{Proof of Theorem \ref{main theorem 2 for all}}
For any integer $M\ge k$, let
$$\mathcal{G}_k(M)=\left\{x\in I: \sum_{i=1}^k\deg A_{n+i}(x)\ge M\textmd{ for all } n\ge0\right\}.$$
Since $\mathcal{G}_k(M+1)\subset \mathcal{G}_k(M)$,
$\lim\limits_{M\to\infty} \dim_{\mathcal{H}}\mathcal{G}_k(M)$ exists.

\begin{lem}\label{upper bound 1/2}
  $\lim\limits_{M\to\infty} \dim_{\mathcal{H}}\mathcal{G}_k(M)\le 1/2.$
\end{lem}

\begin{proof}
  If $x\in \mathcal{G}_k(M)$, we have $\deg A_{ik+1}(x)+\deg A_{ik+2}(x)+\cdots+\deg A_{(i+1)k}(x)\ge M$ for all $i\ge 0$. So for any $n\ge1$,
  $$\Gamma_n:=\{I(A_1,A_2,\ldots,A_{nk}): \sum\limits_{j=ik+1}^{(i+1)k}\deg A_j\ge M\textmd{ for all }0\le i\le n-1\}$$
  is a cover of $\mathcal{G}_k(M)$. For any $t>1/2$, we have
  $$\begin{aligned}
  \mathcal{H}^t(\mathcal{G}_k(M))&\le \liminf\limits_{n\to \infty}\sum\limits_{I(A_1,\ldots,A_{nk})\in\Gamma_n}|I(A_1,\ldots,A_{nk})|^t\\
  &\le \liminf\limits_{n\to \infty}\sum\limits_{I(A_1,\ldots,A_{nk})\in\Gamma_n}q^{-2t\sum_{i=1}^{nk}\deg A_i}\\
  &=\liminf\limits_{n\to \infty}\left(\sum\limits_{I(A_1,\ldots,A_{k})\in\Gamma_1}q^{-2t\sum_{i=1}^{k}\deg A_i}\right)^n.
  \end{aligned}$$
  By Lemma \ref{measure-lem}, it follows that
   $$\begin{aligned}
  \mathcal{H}^t(\mathcal{G}_k(M))&\le \liminf\limits_{n\to \infty}\left(\sum\limits_{m=M}^{\infty}{m-1 \choose k-1}(q-1)^k q^m q^{-2tm}\right)^n\\
  &\le \liminf\limits_{n\to \infty}\left(\sum\limits_{m=M}^{\infty}m^k(q-1)^k q^{(1-2t)m}\right)^n.
  \end{aligned}$$
Since $t>1/2$ and $\sum\limits_{m=1}^{\infty}m^k(q-1)^k q^{(1-2t)m}<\infty$, we have
$\sum\limits_{m=M}^{\infty}m^k(q-1)^k q^{-2tm}< 1$ for all $M$ large enough. It follows that $\mathcal{H}^t(\mathcal{G}_k(M))=0$ and
$\dim_{\mathcal{H}}\mathcal{G}_k(M)\le t$ for all $M$ large enough. Thus
$$\lim\limits_{M\to\infty} \dim_{\mathcal{H}}\mathcal{G}_k(M)\le t.$$
Since $t>1/2$ is arbitrary, we get the desired result.
\end{proof}


\begin{proof}[Proof of Theorem \ref{main theorem 2 for all}]
When $k=1$, it is Theorem 2.3 in \cite{Hu-Wang-Wu-Yu}. We assume $k\ge2$ and distinguish three cases.
\medskip

\noindent {\bf Case 1.} $a=1$. For any $M\ge k$, let $N\ge1$ be the smallest integer such that $\Phi(n)\ge M$ for all $n\ge N$. Then
\[\mathcal{G}_k(\Phi)\subseteq \{x\in I: \deg A_{n+1}(x)+\cdots+\deg A_{n+k}(x)\ge M, n\ge N\}. \]
For any $N$th cylinder $I(A_1,\ldots,A_N)$, let
$$f: \mathcal{G}_k(M)\to I(A_1,\ldots,A_N)\cap T^{-N}(\mathcal{G}_k(M))$$ be defined by
$$f(x)=[A_1,\ldots, A_{N-1}, A_N+x].$$
Then we have \[f(x)=\frac{P_N+xP_{N-1}}{Q_N+xQ_{N-1}}\]
with $P_N/Q_N=[A_1,A_2,\ldots,A_N]$ and hence
\[|f(x)-f(y)|_{\infty}=\frac{|x-y|_{\infty}}{|Q_N^2|_\infty}.\]
Demonstrating that $f$ is a bi-Lipschitz map.
Note that
$$\begin{aligned}
  &\{x\in I: \deg A_{n+1}(x)+\cdots+\deg A_{n+k}(x)\ge M, n\ge N\}\\
  =&\bigcup\limits_{\deg A_1\ge1, \ldots,\deg A_N\ge1}I(A_1,\ldots,A_N)\cap T^{-N}\mathcal{G}_k(M).
\end{aligned}
$$
Since Hausdorff dimension is countably stable and invariant under a bi-Lipschitz map, we have
\[\dim_{\mathcal{H}}\{x\in I: \deg A_{n+1}(x)+\cdots+\deg A_{n+k}(x)\ge M, n\ge N\}=\dim_{\mathcal{H}}\mathcal{G}_k(M) .\]
 So
\[\dim_{\mathcal{H}}\mathcal{G}_k(\Phi)\le \dim_{\mathcal{H}}\mathcal{G}_k(M) \]
for any $M\ge k$. By Lemma \ref{upper bound 1/2}, it follows that \[\dim_{\mathcal{H}}\mathcal{G}_k(\Phi)\le 1/2.\]
For the lower bound, note that
\[ \mathcal{G}_1(\Phi)\subseteq \mathcal{G}_k(\Phi).\]
Since $\dim_{\mathcal{H}}\mathcal{G}_1(\Phi)=1/2$ in this case, we have $\dim_{\mathcal{H}}\mathcal{G}_k(\Phi)\ge 1/2$.
%
%
\medskip

\noindent {\bf Case 2.} $1<a<\infty$. Since
\[ \mathcal{G}_1(\Phi)\subseteq \mathcal{G}_k(\Phi),\]
 we have
\[\dim_{\mathcal{H}}\mathcal{G}_k(\Phi)\ge \dim_{\mathcal{H}} \mathcal{G}_1(\Phi)=\frac{1}{1+a}.\]
For any $\epsilon>0$,
\[\mathcal{G}_k(\Phi)\subseteq \mathcal{F}_k(\Psi)\]
with $\Psi(n)=(a-\epsilon)^n$.
 By Theorem \ref{main-dimension-thm}, we have
$$\dim_{\mathcal{H}}\mathcal{G}_k(\Phi)\le \dim_{\mathcal{H}}\mathcal{F}_k(\Psi)=\frac{1}{1+a-\epsilon}.$$
Letting $\epsilon\to 0$, the conclusion follows in this case.
\medskip

\noindent {\bf Case 3.} $a=\infty$. Then for any $M>1$, we have
\[\mathcal{G}_k(\Phi)\subseteq  \mathcal{F}_k(\Psi)\]
with $\Psi(n)=M^n$. By Theorem \ref{main-dimension-thm}, it follows that
\[\dim_{\mathcal{H}}\mathcal{G}_k(\Phi)\le \mathcal{F}_k(\Psi)=\frac{1}{1+M}.\]
Letting $M\to \infty$, we get the conclusion.
\end{proof}

\providecommand{\bysame}{\leavevmode\hbox to3em{\hrulefill}\thinspace}
\providecommand{\MR}{\relax\ifhmode\unskip\space\fi MR }
\providecommand{\MRhref}[2]{%
  \href{http://www.ams.org/mathscinet-getitem?mr=#1}{#2}
}
\providecommand{\href}[2]{#2}

 \end{document}